\documentclass[12pt]{article}
\usepackage{amsfonts,amscd,amssymb}
\usepackage{theorem}
\newtheorem{definition}{Definition}[section]
\newtheorem{lemma}[definition]{Lemma}

{\theorembodyfont{\rmfamily}}
{\theorembodyfont{\rmfamily}}
\newtheorem{theorem}[definition]{Theorem}
{\theorembodyfont{\rmfamily}}
{\theorembodyfont{\rmfamily}}
{\theorembodyfont{\rmfamily}}

\newenvironment{proof}{{\it Proof.}}{\hfill $ \square $ \vskip 4mm}

\begin{document}

\title{Class numbers, Ono invariants and some interesting primes}

\author { Alexandru GICA}
\date{University of Bucharest, Faculty of Mathematics and Informatics\\Str. Academiei 14, Bucharest 1, Romania RO-010014\\ e-mail address: alexandru.gica@unibuc.ro}
\maketitle \small{ {\bf Abstract:}  
Our aim is to find all the prime numbers $p$ such that $p+x^2$ has at most two different prime factors, for all the odd integers $x$ such that $x^2<p$. We solve entirely the cases $p\equiv 1,3,5 \pmod 8$, using the knowledge of the quadratic imaginary number fields with class numbers 4,1 and 2 respectively.  The case $p\equiv 7 \pmod 8$ is not completely solved. Taking into account a result of St\'{e}phane Louboutin, we prove that there is at most one value $p\equiv 7 \pmod 8$ besides our list. Assuming a Restricted Riemann Hypothesis, the list is complete.

\textit{Key words and phrases:} Class numbers, Ono invariants.

\textit{Mathematics Subject Classification (2020):} 11R29, 11R11.

\section{Introduction} 
$\;\;\;$The main goal of this paper is to prove the following.
\begin{theorem}
Let $p$ be an odd prime such that $p+x^2$ has at most two different prime factors for any odd integer $x$ such that $x^2<p$. If $p\equiv 5 \pmod 8$, then $p=5,13,37$. If $p\equiv 1 \pmod 8$, then $p=17,73,97,193$. If $p\equiv 3 \pmod 8$, then $p=3,11,19, 43, 67, 163$.  If $p\equiv 7 \pmod 8$ then with one possible exception $p=7,23,31,47,79,103,127,151,223,463,487,$ $823,1087,1423.$ Assume the Restricted Riemann Hypothesis: $\zeta _K(1-(2/log d_K))\leq 0$ for the Dedekind zeta functions $\zeta_K(s)$ of all imaginary quadratic fields $K$. Then, the above list in the case $p\equiv 7 \pmod 8$ is complete.
\end{theorem}

$\;\;\;$We will use the standard notations $\omega$ and $\Omega$. If $n=p_1^{a_1}p_2^{a_2}...p_r^{a_r}$ is the standard decomposition of $n$ as product of primes (this means that $p_j$ are different prime numbers and $a_j$ are positive integers), then $\Omega (n)=\displaystyle\sum _{j=1}^r a_j$ and $\omega (n)=r$ (we consider $\Omega (1)=\omega (1)=0$). We call $\Omega (n)$ the length of the positive integer $n$. The prime numbers $p$ from Theorem 1.1 have the property $\omega (p+x^2)\leq 2$, for any odd integer $x$ such that $x^2<p$. 

$\;\;\;$ Some years ago, we tried to find another proof of the class one problem, based on the notion of length. We posed the following.

{\bf Strong Conjecture:} If $n>163$ is a positive integer, there exist the positive integers $a,b$ such that $n=a+b$, the length $\Omega(ab)$ is an even number and any prime divisor of $ab$ is smaller than $n/4$.

The Strong Conjecture implies the class one problem (there are only nine quadratic imaginary fields with class number one). While we don't succeed to prove the Strong Conjecture, we were able the prove a weaker version of this conjecture (see [5]).

\begin{theorem}
If $n>3$ is a positive integer, there exist the positive integers $a,b$ such that $n=a+b$ and the length $\Omega(ab)$ is an even number.
\end{theorem} 

There is an interesting interplay between class numbers and lengths. We wrote the chapter \textit{Lengths and Class Numbers} for the collective volume \textit{New Frontiers in Number Theory and Applications} (which will be published by Springer Publishing House), in which we described this interplay between class numbers and lengths. Theorem 1.1 shows that there is also a connection between class numbers and the function $\omega$.

$\;\;\;$ We divided the proof of Theorem 1.1 in four cases. The case $p\equiv 1 \pmod 4$ has an extra difficulty since Minkowski's constant for the field $K=\mathbb{Q}(i\sqrt p)$ (in this case) is $\frac{4}{\pi}\sqrt p>\sqrt p$. In the case $p\equiv 5 \pmod 8$ we prove that the class number of $K$ is 2. Using the work of Baker (see [3]), Stark (see [11]), Montgomery and Weinberger (see [7]) we are able to solve this case. In the case $p\equiv 1 \pmod 8$ we prove that the class number of $K$ is 4. Using the work of Arno (see [1]), we are able to settle this case. In the case $p\equiv 3 \pmod 8$ we prove that the class number of $K$ is 1. Using the work of Baker (see [2]) and Stark (see [10]), we are able to resolve this case. The case $p\equiv 7 \pmod 8$ is the most difficult. We prove that in this case the Ono invariant of $K$ equals the class number of the field $K$. Using the work of Louboutin (see [6]), we were able to solve this case, but we are not sure that the list is complete: 
$$p=7,23,31,47,79,103,127,151,223,463,487,823,1087,1423.$$ There is at most one possible value of $p$ besides this list. If we assume the Restricted Riemann Hypothesis from the statement, it follows that the list is complete.

\section{The case $p\equiv 5 \pmod 8$}

$\;\;\;$ We consider the number field $K=\mathbb{Q}(i\sqrt p)$, where $p\equiv 1 \pmod 4$ is a prime number. In this case, the ring of the integers of $K$ is $A=\mathbb{Z}[i\sqrt p]$. The discriminant is $\delta_K=-4p$ and we have $2A=P_1^2$, where $P_1$ is a maximal ideal with norm 2. Since the equation $x^2+py^2=2$ has no integer solutions, it is obvious that $P_1$ is not a principal ideal and the ideal class group of $K$ is not trivial. 

$\;\;\;$When we try to compute the ideal class group for the field $K$, we have to consider all its generators: the nonprincipal maximal  ideals of $A$ with norm smaller than Minkowski's constant, which is (in this case) $\frac{4}{\pi}\sqrt p$. Since $\frac{4}{\pi}>1$, we will have a problem with the maximal ideals $Q$ with norm between $\sqrt p$ and $\frac{4}{\pi}\sqrt p$. Fortunately, for primes with the property stated in our theorem, the ideals $Q$ as above does not bother us. 

\begin{lemma}
Let $p\equiv 1 \pmod 4$ be a prime number such that $p\geq 137$ and $\omega (p+x^2)\leq 2$, for any odd integer $x$ with $x^2<p$. Then the ideal class group of the field $K=\mathbb{Q}(i\sqrt p)$ is generated by the classes of the maximal ideals with norms smaller than $\sqrt p$. 
\end{lemma}

\begin{proof} Let us consider a maximal nonprincipal ideal $Q$ of the ring $A=\mathbb{Z}[i\sqrt p]$  with the norm satisfying the inequalities $$\sqrt p< N(Q)<\frac{4}{\pi}\sqrt p.$$  The norm of the ideal $Q$ should be a prime number $q$: $N(Q)=q$. Using standard facts about ramification in quadratic number fields, $-p$ should be a quadratic residue modulo $q$: $\Big(\frac{-p}{q}\Big)=1$. There is a positive odd integer $x$ such that $x<q$ and $-p\equiv x^2 \pmod q$. We have the equality $p+x^2=2qa$, where $a$ is a positive odd integer. If $a=1$, we reach a contradiction:
$$3\sqrt p<p<p+x^2=2q<2\times \frac{4}{\pi}\sqrt p<3\sqrt p;$$ we used the fact that $p>9$. This means that $a>1$. The next step is to prove that $a$ is a prime number. If this is not true, then $\Omega(p+x^2)\geq 4$ and considering $r$, the smallest odd prime divisor of the number $p+x^2$, we have the inequality $p+x^2\geq 2r^3.$ Since $x<q<\frac {4}{\pi}\sqrt p<\frac{4}{3}\sqrt p$, we have 
$$2r^3\leq p+x^2<p+\frac{16}{9}p=\frac{25}{9}p$$ and $r<\sqrt [3] {\frac{25p}{18}}<\sqrt p$. Since $-p$ is a quadratic residue modulo $r$, we find a positive odd integer $y$ such that $y<r$ and $r$ divides $p+y^2$. Since $y^2<r^2<p$, the assumption of the theorem applies and $p+y^2=2r^b$, for a positive integer $b$. We have $$(-y+2r)^2<4r^2<4 \sqrt [3] {\frac{625p^2}{324}}<p;$$ the last inequality holds since $p\geq 137$. Since $(2r-y)^2<p$ and $(2r-y)^2\equiv y^2\equiv -p \pmod r$, we get (using the assumption of the theorem) that 
$$(2r-y)^2+p=2r^c,$$ for a positive integer $c$. Since $2r-y>y$, it follows that $c>b$. If $b\geq 2$, we should have that $r^2$ divides $(2r-y)^2-y^2=4r(r-y)$; contradiction. Therefore $b=1$ and 
$$2\sqrt p>2r=p+y^2>p;$$
 this is again a contradiction. We get that indeed $p+x^2=2qa$ and $a$ is a prime number. If $a\geq q$ we obtain a contradiction since $$p+q^2>p+x^2=2qa\geq 2q^2$$ (do not forget that $q^2>p$). It follows that $a<q$. We consider now the principal ideals $(x+i\sqrt p)A$ and $(x-i\sqrt p)A$. One of them is a multiple of the ideal $Q$ and we have an equality like $$(x\pm i\sqrt p)A=P_1QR,$$ where $R$ is a maximal ideal with norm $a<q$. We proved that the class of $Q$ in the ideal class group of $K$ is the same as the class of $P_1\times R^d$, where $N(R)<q$. Using now a descent argument, it follows that the class of $Q$ is the class of a product of maximal ideals with norms smaller than $\sqrt p$. \end{proof} 

$\;\;\;$ It looks a little bit tiresome, but for now one, in the case $p\equiv 1 \pmod 4$, we can analyze only the maximal ideals with norms smaller than $\sqrt p$. We are now ready to prove the theorem in the case $p\equiv 5 \pmod 8$.

\begin{proof} We suppose that $p\equiv 5 \pmod 8, p\geq 149$, is a prime number with the properties asserted in the theorem (the inequality will allow us to use the previous Lemma). We will show that in this case the class number for the field $\mathbb{Q}(i \sqrt p)$ is two. 
The first step is to prove that $p+x^2$ is twice a prime number for any odd integer $x$ such that $x^2<p$. Let us consider an odd positive integer $x$ such that $x^2<p$. According to the assumption of the theorem, $p+x^2=2q^a$, where $q$ is a prime number and $a$ is a positive integer. Let us suppose, for the moment, that $a\geq 3$. Then 
$$2p>p+x^2=2q^a\geq 2q^3.$$
From this inequality it follows that $q<\sqrt [3] p$. Since $-p$ is a quadratic residue modulo $q$, we can find an odd positive integer $y$ such that $y<q$ and $$p+y^2=2q^b,$$ where $b$ is a positive integer. It is easy to see that 
$(2q-y)^2<4q^2<4\sqrt [3] {p^2}<p$ (the last inequality is true since $p>64$). This means that we can apply the assumption of the theorem for the odd positive integer $2q-y$: $$p+(2q-y)^2=2q^c$$ where $c$ is a positive integer. Since $2q-y>y$, it follows that $c>b$. If we subtract the above two equalities, we get 
$$2q^b(q^{c-b}-1)=(2q-y)^2-y^2=4q(q-y).$$ From here we infer that $b=1$ and $p+y^2=2q$. The contradiction is straightforward since 
$$\sqrt [3] p>q=\frac{p+y^2}{2}>\frac{p}{2}.$$ We proved that $a\leq 2$. The possibility $a=2$ does not occur since the equality $p+x^2=2q^2$ would imply that $2$ is a quadratic residue modulo $p$: this is not true since $p\equiv 5 \pmod 8$. Hence $a=1$ and the first step of the proof is complete.

The second step of the proof is to show that the class number of the field $\mathbb{Q}(i \sqrt p)$ is two. According to the Lemma 2.1 we have to consider only odd prime numbers $q<\sqrt p$, for which $-p$ is a quadratic residue modulo $q$. There exists an odd positive integer $x<q$ such that $q$ divides $p+x^2$. From this (taking into account the inequality $x^2<q^2<p$, the assumption of the theorem and the first step of the proof) it follows that $p+x^2=2q$. This is not possible since $$\sqrt p>q=\frac{p+x^2}{2}>\frac{p}{2}$$ and $p>4$. We proved that the ideal class group for the field $\mathbb{Q}(i \sqrt p)$ is generated by $P_1$, the maximal ideal of $A$ for which $2A=P_1^2$ (see the beginning of this section for the definition of the ideal $P_1$). We saw above that $P_1$ is not principal and this means that the class number for the field $\mathbb{Q}(i \sqrt p)$ is two. We call now the result of Stark (see [11]), Montgomery and Weinberger (see [7]). There are only 18 quadratic imaginary fields with class number two. Looking in the list of this numbers, we see that only three values fit the conditions of our theorem:
$$p=5,13,37.$$

\end{proof}

\section{ The case $p\equiv 1 \pmod 8$}

\begin{proof} We suppose that $p\geq 137$ is a prime number with the properties asserted in the theorem (the inequality will allow us to use Lemma 2.1). We will show that in this case the class number for the field $\mathbb{Q}(i \sqrt p)$ is four (and the ideal class group is isomorphic with $\mathbb{Z}_4$). 
In the same way as in the second section we prove that $p+x^2$ is twice a prime number or twice the square of a prime number, for any odd integer $x$ such that $x^2<p$. 

The second step of the proof is to show that the class number of the field $\mathbb{Q}(i \sqrt p)$ is four. According to Lemma 2.1 we have to consider only odd prime numbers $q<\sqrt p$, for which $-p$ is a quadratic residue modulo $q$. There exists an odd positive integer $x<q$ such that $q$ divides $p+x^2$. From this (taking into account the inequalities $x^2<q^2<p$, the assumption of the theorem and the first step of the proof) it follows that $p+x^2=2q$ or $p+x^2=2q^2$. The first case is not possible since $$\sqrt p>q=\frac{p+x^2}{2}>\frac{p}{2}$$ and $p>4$. We have to study now the case $p+x^2=2q^2$. We have to point out that for any prime $p\equiv 1 \pmod 8$ there exists an odd positive integer $x$ such that $x^2<p$ and $p+x^2=2y^2$. According to the above arguments, $y$ should be an odd prime $q$. For this prime $q$ we have $2A=Q_1\times Q_2$, where $Q_1,Q_2$ are different maximal ideal of norm $q$. We are now looking to the decomposition as product of maximal ideals for the principal ideal $(x+i\sqrt p)A$ where $x$ is the positive odd integer for which $p+x^2=2q^2$. It is easy to see that we have 
$$(x+i\sqrt p)A=P_1\times Q_j^2,$$ where $j=1$ or $j=2$. Since the order of the class of $P_1$ is 2 (in the ideal class group of the field $K$), we get that the order of the class of $Q_1$ is 4 in the ideal class group of the field $K$.
We proved that the ideal class group for the field $\mathbb{Q}(i \sqrt p)$ is generated by $Q_1$, the maximal ideal of $A$ for which $qA=Q_1\times Q_2$.  We call now the result of Arno (see [1]). There are only 54 quadratic imaginary fields with class number four. Looking in the list of this numbers, we see that only four values fit the conditions of our theorem:
$$p=17,73,97,193.$$
\end{proof}

\section{The case $p\equiv 3 \pmod 8$}

In this case we prove that the ring of the integers of the number field $\mathbb{Q}(i\sqrt p)$, $A=\mathbb{Z}[\frac{1+i\sqrt p}{2}]$, is a UFD. Stark and Baker find out all such quadratic imaginary fields. The values which fit our problem are $$p=3,11,19,43,67,163.$$ In the sequel $p\equiv 3 \pmod 8$ is a prime number $p\geq 19$, such that $\omega(p+x^2)\leq 2$, for all the odd integers $x$ such that $x^2<p$.

\begin{proof} The Minkowski's constant in this case is $\frac{2}{\pi}\sqrt p<\sqrt p$. Since $p\equiv 3 \pmod 8$, $2$ is inert (that is, $2A$ is a maximal ideal). Let us consider now an odd prime $q<\sqrt p$. Suppose that the Legendre symbol $\Big(\frac{-p}{q}\Big)$ is 1. This means that there is an odd positive integer $x$ such that $x<q$ and $-p \equiv x^2 \pmod q$. Since $x^2<q^2<p$, the assumption of the theorem ensure us that 
$$x^2+p=4q^a,$$ where $a$ is a positive integer. Let us suppose that $a\geq 3$. This means that $$4q^3\leq 4q^a=x^2+p<2p$$ and $q<\sqrt [3] {\frac{p}{2}}$. Since $p\geq 19>16$, we have
$$(2q-x)^2<4q^2<4\sqrt [3] {\frac{p^2}{4}}<p.$$ By our assumption,
$$(2q-x)^2+p=4q^b,$$ where $1\leq a<b$. In fact, $a=1$ since for $a\geq 2$ then $q^2$ should divides the difference $(2q-x)^2-x^2=4q(q-x)$; this is impossible. It follows from the above argument that $a=1$ or $a=2$. If $a=2$, then 
$$p=4q^2-x^2=(2q-x)(2q+x).$$
Since $p$ is a prime number, we get $2q-x=1$, $2q+x=p$ and $q=\frac{p+1}{4}$. We obtain the inequalities 
$$\frac{p+1}{4}=q<\sqrt p$$ and 
$$0>p-4\sqrt p+1=\sqrt p (\sqrt p -4)+1>1,$$ since $p>16$. We reached a contradiction. We get that $a=1$ and $p+x^2=4q$. We obtain again a contradiction: $$\sqrt p>q=\frac{x^2+p}{4}>\frac{p}{4}.$$ This means that for all the odd prime numbers $q$ smaller than the Minkowski's constant, $-p$ is a quadratic nonresidue modulo $q$; this implies that $qA$ is a maximal ideal for any such a prime $q$. We proved that $A$ is a principal ring. Using now the result of Baker (see [2]) and Stark (see [10]) we succeed to prove that the only solutions in this case are the prime numbers $3,11,19, 43,67,163$.\end{proof}

\section{The case $p\equiv 7 \pmod 8$}
\begin{proof} We suppose that $p\geq 47, p\equiv 7 \pmod 8$ is a prime number with the properties stated in Theorem 1.1. In this case, the ring of integers of the number field $K=\mathbb{Q}(i\sqrt p)$ is $A=\mathbb{Z}[\frac{1+i\sqrt p}{2}]$. Since $p\equiv 7 \pmod 8$, we have the equality $2A=P_1P_2$, where $P_1,P_2$ are two different maximal ideals with norm 2.

{\bf First step:} We prove that the ideal class group of $K$ is generated by the class of the ideal $P_1$. We know that Minkowski's constant in this case is $\frac{2}{\pi}\sqrt p<\sqrt p$. Let us consider an odd prime $q$ such that $-p$ is a quadratic residue modulo $q$ and $q<\sqrt p$. There is a positive odd integer $x$ such that $x<q$ and $q$ divides $p+x^2$. According to the properties of $p$, we have $p+x^2=2^aq^k$. We will prove that $k=1$ or $k=2$ Let us suppose that $k\geq 3$. Then 
$$2p>p+x^2\geq 8q^3$$ and $q<\sqrt [3] {\frac{p}{4}}$.
We have the following three equalities:
\begin{equation}p+x^2=2^aq^k, \end{equation}
\begin{equation}p+(2q-x)^2=2^bq^l,\end{equation}
\begin{equation}p+(2q+x)^2=2^cq^m,\end{equation} where $a,b,c,k,l,m$ are positive integers and $a,b,c\geq 3$. The assumption of the theorem holds since $(2q+x)^2<(3q)^2<9\sqrt [3] {\frac{p^2}{16}}<p$ (the last inequality hold since $p\geq 47>\frac{729}{16}$). We supposed that $k\geq 3$. Subtracting the equation (1) from the equation (2) we get $4q(q-x)=2^bq^l-2^aq^k$. It follows that $l=1$ (since $4q(q-x)$ is not a multiple of $q^2$). Subtracting the equation (1) from the equation (3)  we obtain, in the same way, that $m=1$ (since $4q(q+x)$ is not a multiple of $q^2$). We obtained
\begin{equation}p+(2q-x)^2=2^bq,\end{equation}
\begin{equation}p+(2q+x)^2=2^cq.\end{equation} Since $2q-x<2q+x$, we have $3\leq b<c$. If we subtract the above two equalities, we get:
$$8qx=q\times 2^b(2^{c-b}-1).$$ It follows that $b=3$ and we obtain the contradiction
$$\sqrt [3] {\frac{p}{4}}>q=\frac{p+(2q-x)^2}{8}>\frac{p}{8}.$$ Hence $k=1$ or $k=2$. We have $p+x^2=2^aq$ or $p+x^2=2^aq^2$. We know that $qA=Q_1Q_2$, where $Q_1,Q_2$ are different maximal ideals of $A$ with norm $q$.
 Consider the first case ($p+x^2=2^aq$). We have $$\Big(\frac{x+i\sqrt p}{2}\Big)A=P_i^{a-2}Q_j,$$ where $i,j$ could be 1 or 2. From here we infer that the class of $Q_1$ in the ideal class group of $K$ is a power of the class of $P_1$. In the second case ($p+x^2=2^aq^2$)  we have $$\Big(\frac{x+i\sqrt p}{2}\Big)A=P_i^{a-2}Q_j^2,$$ where $i,j$ could be 1 or 2. From here we infer that the class of $Q_1^2$ in the ideal class group of $K$ is a power of the class of $P_1$. Since the ideal class group of $K$ has an odd cardinal (in the case $p\equiv 7 \pmod 8$), it follows immediately that the class of $Q_1$ in the ideal class group of $K$ is a power of the class of $P_1$. We finished the proof of the first step.

{\bf Second step:} We prove that $d$, the Ono invariant of the field $K$, equals $h$, the class number of the field $K$. We remind the reader the definition of the Ono invariant in the case $K=\mathbb{Q}(i\sqrt p)$ and $p$ is a prime number such that $p\equiv 7 \pmod 8$. The Ono invariant $d$ of the field $K$ (usually denoted by $\textit{Ono}_K$) is defined by the equality
$$d=max \{\Omega (\frac{(2n+1)^2+p}{4}), 0 \leq n \leq \frac{p-7}{4}\}.$$ R. Sasaki  (see [9]) proved that we have always the inequality $$d\leq h.$$ This result will be helpful in the proof of the second step.

$\;\;\;$ It is elementary to prove that there is an odd positive integer $x$ such that $p+x^2=8t^2$ and $x^2<p$ (see[5], the proof of Theorem 1). Let us suppose that $t=2^u$. Then $$\Big(\frac{x+i\sqrt p}{2}\Big)A=P_i^{hs},$$ where $i$ is 1 or 2 and $hs=2u+1$ (we used above the first step: the ideal class group for $K$ is generated by the class of $P_1$). From the above equalities it follows (since $\frac{x-1}{2}\leq \frac{p-7}{4}$; it is easy to check this) that 
$$d\geq 2u+1=hs\geq h.$$ Since the inequality $d\leq h$ is always true, we get $h=d$. Suppose now that $t$ has at least one odd prime divisor $q$. In this case $p+x^2=8q^2$ and $4q^2<p$ (we will prove this in the end of this step). From the assumption of our theorem, 
$$p+(2q-x)^2=4q(3q-x)$$ and $3q-x=2^u$.
 We have the equality
$$p+(-3x+8q)^2=8(3q-x)^2=2^{2u+3}.$$
Then $$\Big(\frac{-3x+8q+i\sqrt p}{2}\Big)A=P_i^{hs},$$ where $i$ is 1 or 2 and $hs=2u+1$ (we used above the first step: the ideal class group for $K$ is generated by the class of $P_1$). From the above equalities it follows (since $\frac{8q-3x-1}{2} \leq \frac{p-7}{4}$; it is easy to check this)
$$d\geq 2u+1=hs\geq h.$$ Since the inequality $d\leq h$ is always true, we get $h=d$. 

$\;\;\;$There is a debt which should be payed: if $t$ has at least one odd prime divisor $q$, then $p+x^2=8q^2$. From the hypothesis of the theorem we know that we have the equality
$$p+x^2=2^aq^{2b},$$ where $a,b$ are positive integers, $a$ is odd and $a\geq 3$.We have to prove that $a=3$ and $b=1$. Suppose that this is not true. Then $2p>32q^2$ and $p>16q^2$. From this inequality and $x^2<p$, it follows that $(2q-x)^2<p$. According to our hypothesis, we have 
$$p+(2q-x)^2=2^cq;$$ the coefficient of $q$ is 1 since $(2q-x)^2-x^2=4q(q-x)$ is not a multiple of $q^2$. In the same way we prove that we have the equality 
$$p+(4q-x)^2=2^dq.$$ Since  $(4q-x)^2-x^2=8q(2q-x)$ is not a multiple of $16$, it follows that $a=3$ (if $d=3$, we reach the contradiction $\frac{\sqrt p}{4}>q>\frac{p}{8}$). According to our assumption, $b\geq 2$ and $2p>72q^2, p>36q^2$. In the same way as above we prove that we have the equality 
$$p+(6q-x)^2=2^eq.$$ Since $(6q-x)^2-(2q-x)^2=4q(8q-2x)$ is not a multiple of 16, it follows that $c=3$ or $e=3$. In both cases we obtain the contradiction
$$\frac{\sqrt p}{6}>q>\frac{p}{8}.$$ We proved that in this case $p+x^2=8q^2$ and the proof of the second step is now complete.

{\bf The third step:}
J. Cohen, J. Sonn (see [4]), F. Sairaji and K. Shimizu (see [8]) proved that there are only finitely many imaginary quadratic number fields $K$ whose Ono invariant are equal to their class number $h_K$. St\'{e}phane Louboutin was able to find, with one possible missing field, all the above quadratic fields. Assuming a Restricted Riemann Hypothesis, the list is complete. Checking the list of Louboutin (see [6], Table 1, page 2293) we find that only 14 values of $p$ fit our situation:
$$p=7,23,31,47,79,103,127,151,223,463,487,823,1087,1423.$$
\end{proof}

{\bf Remarks:}

{\bf 1)} The list of Louboutin contains 114 values. Only 24 of these values correspond to prime numbers $p\equiv 7 \pmod 8$. From these 24 values, only 14 fit with the condition of Theorem 1.1. This opens a possible way for finding a complete proof in this case (which don't use the result of Louboutin).

{\bf 2)} Another possibility to tackle Theorem 1.1 is to use an odd prime $q$ which is a quadratic residue modulo $p$ and $q<\frac{\sqrt p}{5}$. We need to know a concrete bound $p_0$ such that for any $p>p_0$ there is an odd prime $q$ with the above properties ($q$ is a quadratic residue modulo $p$ and $q<\frac{\sqrt p}{5}$).

\
{\large \bf References}

[1] S. Arno, \textit{The imaginary quadratic fields of class number 4}, Acta Arith. {\bf 60} (1992), no. 4, 321--334. 

[2] A. Baker, \textit{Linear forms in the logarithms of algebraic numbers I}, Mathematika, {\bf 13}(1966), 204--216.

[3] A. Baker, \textit{Imaginary quadratic fields of class number 2}, Ann. of Math., {\bf 94}(1971), 139--152.

[4] J. Cohen, J. Sonn, \textit{On the Ono invariants of imaginary quadratic fields}, J. Number Theory, {\bf 95}(2002), 259--267.

[5] A. Gica, \textit{The Proof of a Conjecture of Additive Number Theory}, J. Number Theory, {\bf 94}(2002), 80--89.

[6] S. Louboutin, \textit{On the Ono invariants of imaginary quadratic number fields}, J. Number Theory, {\bf 129}(2009), 2289--2294.

[7] H.L. Montgomery, P.J. Weinberger, \textit{Notes on small class numbers}, Acta Arith., {\bf 24}(1974), 529--542.

[8] F. Sairaji, K. Shimizu, \textit{An inequality between class numbers and Ono's numbers associated to imaginary quadratic fields}, Proc. Japan Acad. Ser. A {\bf 78} (2002), 105--108.

[9] R. Sasaki, \textit{On a lower bound for the class number of an imaginary quadratic field}, Proc. Japan Acad. Ser. A {\bf 62} (1986), 37--39.

[10] H.M. Stark, \textit{A complete determination of the complex quadratic fields of class number 1}, Michigan Math. J. {\bf 14} (1967), 1--27.

[11] H.M. Stark, \textit{On Complex Quadratic Fields with Class-Number Two}, Math. Comput. {\bf 29} (1975), 289--302.

\end{document}